\documentclass[11.5pt,regno]{amsart}
\usepackage{amsmath,amsfonts,amssymb,hyperref,mathrsfs,esint}
\usepackage{amssymb}
\usepackage{amsfonts}
\usepackage{amsthm}
\usepackage{amsmath}
\usepackage{bbm}
\usepackage{mathtools}
\usepackage{color}

\DeclareMathOperator{\supp}{supp}

\DeclareMathOperator{\ptt}{Part}
\numberwithin{equation}{section}
\allowdisplaybreaks
\newcommand{\ep}{{\epsilon}}
\newtheorem{theorem}{Theorem}[section]
\newtheorem{proposition}[theorem]{Proposition}

\newtheorem{lemma}[theorem]{Lemma}
\newtheorem{corollary}[theorem]{Corollary}

\theoremstyle{definition}
\newtheorem{definition}[theorem]{Definition}

\theoremstyle{remark}

\newcommand{\R}{\mathbb{R}}

\newcommand{\C}{\mathbb{C}}

\renewcommand{\hat}{\widehat}

\newcommand{\abs}[1]{\left\vert#1\right\vert}

\newcommand{\scriptN}{\mathcal{N}}

\DeclarePairedDelimiter{\norm}{\lVert}{\rVert}

\def\XXint#1#2#3{{\setbox0=\hbox{$#1{#2#3}{\int}$ }
		\vcenter{\hbox{$#2#3$ }}\kern-.6\wd0}}

\begin{document}
\title{$\ell^2$ Decoupling in $\R^2$ for curves with Vanishing Curvature}

\author{Chandan Biswas}
\address{Mathematical Sciences Department, University of Cincinnati, Cincinnati, OH 45221, USA}
\email{chandan.biswas@uc.edu}

\author{Maxim Gilula}
\address{Department of Mathematics, Michigan State University, East Lansing, MI 48824, USA}
\email{gilulama@math.msu.edu}

\author{Linhan Li}
\address{Department of Mathematics, Brown University, Providence, RI 02912, USA}
\email{linhan$\_$li@brown.edu}

\author{Jeremy Schwend}
\address{Department of Mathematics, University of Wisconsin, Madison, WI 53706, USA}
\email{jschwend@math.wisc.edu}

\author{Yakun Xi}
\address{Department of Mathematics, University of Rochester, Rochester, NY 14620, USA}
\email{yxi4@math.rochester.edu}

\thanks{This material is based upon work supported by the National Science Foundation under Grant No. 1641020. We would also like to thank the AMS and everyone who helped make possible the Oscillatory Integrals Mathematics Research Community held in June of 2018. Schwend was also supported by NSF DMS-1653264 and DMS-1147523.}

\begin{abstract}
	We expand the class of curves $(\varphi_1(t),\varphi_2(t)),\ t\in[0,1]$ for which the $\ell^2$ decoupling conjecture holds for $2\leq p\leq 6$. Our class of curves includes all real-analytic regular curves with isolated points of vanishing curvature and all curves of the form $(t,t^{1+\nu})$ for $\nu\in (0,\infty)$.
\end{abstract}

\maketitle

\section{Introduction}

Let $g$ be a locally integrable function defined on a measurable set $Q$ in $\R$, and define the $(\varphi_1,\varphi_2)$ extension operator by $$E^{\varphi_1,\varphi_2}_Qg(x_1,x_2)=\int_Q e((x_1,x_2)\cdot(\varphi_1(t),\varphi_2(t)) g(t)dt,$$ where $e(z)=e^{2\pi i z}.$ We will just write $E_Q g(x_1,x_2)$ from now on because $\varphi_1$ and $\varphi_2$ will be clear from the context.

This extension operator with $Q=[0,1]$ and the curve $(\varphi_1, \varphi_2)$ with vanishing or infinite curvature is the main object of study in this paper. In particular, we prove an $\ell^2$ decoupling inequality of the form

\begin{equation}
\|E_{[0,1]}g\|_{L^p(\mathbb R^2)}\le C_\epsilon \delta^{-\epsilon}\left( \sum_{\Delta\in \ptt_{\delta^{1/2}}([0,1])}\|E_\Delta g\|_{L^p(\mathbb R^2)}^2\right)^{1/2}\label{main1}
\end{equation}
for curves with curvature that vanishes, or goes to infinity at finite order. Above, $\ptt_{\delta^{1/2}}([0,1])$ denotes a partition of the unit interval into subintervals of size $\delta^{1/2}$, and $2\le p\le 6.$

The notorious difficulty of this problem for various extension operators has led to few developments in the 2000's since Wolff \cite{Wo00}, but just a few years ago in the revolutionary paper \cite{bourgain2015proof}, Bourgain-Demeter proved inequality \eqref{main1} for  the curve $(t,t^2)$ in $\R^2$, as well as generalizations to all curves and hypersurfaces with non-vanishing curvature. This $\ell^2$ decoupling inequality of Bourgain-Demeter had many powerful applications; for example, the proof of the main conjecture in Vinogradov's Mean Value Theorem   \cite{bourgain2016proof}, an 80 year old problem in number theory counting integer solutions to a system of equations of the form $x_1^k+\cdots + x_s^k=y_1^k+\cdots+y_s^k$, used $\ell^2$ decoupling as a key tool. One should consult Pierce's exposition \cite{pierce2017vinogradov} on the Vinogradov Mean Value Theorem for an almost complete list of recent advances and references regarding Bourgain-Demeter's result.

The $\ell^2$ decoupling theory of Bourgain-Demeter helped make significant progress on, and even close, multiple long standing open problems in harmonic analysis and number theory, and has the potential of even further significant applications. This prompts one to extend the decoupling inequalities to include a bigger class of curves such as those  with vanishing curvature, which we study in this paper. A model case to keep in mind is the  curve $(\varphi_1(t), \varphi_2(t))=(t,t^{1+\nu}),\ \nu>0$, where the curvature is 0 or $\infty$ at the origin.  Our result for this model case is the following.
\begin{theorem}\label{model}
	Given the curve  $(t,t^{1+\nu})$ for fixed $\nu>0$, for all $2\leq p \leq 6$, all $\epsilon>0$, and all $g:[0,1]\to\C$ there is a constant $C_\epsilon$ such that
\begin{equation}
\|E_{[0,1]}g\|_{L^p(\mathbb R^2)}\le C_\epsilon \delta^{-\epsilon}\left( \sum_{\Delta\in \ptt_{\delta^{1/2}}([0,1])}\|E_\Delta g\|_{L^p(\mathbb R^2)}^2\right)^{1/2}.
\end{equation}

\end{theorem}
Next, for the more general version with curve $(t, \varphi(t))$:

\vspace{.5pc}

First, if $\varphi''$ is in $C^{0,\alpha}([0,1])$, some $\alpha>0$, and bounded away from $0$ and $\infty$, the decoupling result automatically follows from the work of section 7 of Bourgain-Demeter \cite{bourgain2015proof}. To make this result more general, at points where one of these is not satisfied we will need two things, as mirrored in the assumptions of Theorem \ref{main} and Proposition \ref{mainprop}:

\vspace{.5pc}
\begin{flushleft}
(1) Control on the rate of vanishing (or blowup) of $\varphi''$.
\\
(2) Control on the nearby $C^{0,\alpha}$ norm, for use in bounding error with Taylor's theorem.
\end{flushleft}

\begin{definition}Let $\varphi''$ be defined near $z\in \R$. Then $r^{+}_{2,S}(z)$, the \textit{supremal right order of vanishing\footnote{We are abusing the word \lq\lq vanishing\rq\rq a bit here since we also include the $\infty$ case, for which the \lq\lq vanishing\rq\rq order is negative.} of $\varphi''$ at $z$}, and $r^{+}_{2,I}(z)$, the \textit{infimal right order of vanishing of $\varphi''$ at $z$}, will be defined as
$$
r^{+}_{2,S}(z):=\sup\left\lbrace s:\limsup_{t\rightarrow z^{+}}\frac{\varphi''(t)}{t^s}=0\right\rbrace,
$$
$$
r^{+}_{2,I}(z):=\sup\left\lbrace s:\liminf_{t\rightarrow z^{+}}\frac{\varphi''(t)}{t^s}=0\right\rbrace,
$$
and similarly for the left order of vanishings. Furthermore, define $r_{3,S}$ similarly for $\varphi'''$. If $\varphi''$ is in H\"older space $C^{0,\alpha}(\Omega)$, $0<\alpha<1$, then define
$$
|\varphi''|_{C^{0,\alpha}(\Omega)}:=\sup_{x\neq y\in \Omega}\frac{|\varphi''(x)-\varphi''(y)|}{|x-y|^\alpha}.
$$
Finally, the \textit{supremal right $C^{0,\alpha}$-order of vanishing of $\varphi''$ at $z$} will be defined as
$$
\overline{r}^{+}_{2+\alpha,S}(z):=\sup\left\lbrace s:\limsup_{t\rightarrow 0^+}\frac{|\varphi''|_{C^{0,\alpha}(\{z\}+[t,2t])}}{t^s}=0\right\rbrace,
$$
and likewise for the left order of vanishing.
\end{definition}

In Proposition \ref{mainprop}, we will reduce Theorem \ref{main} to the case where $z=0$, and $\varphi$ is defined on an interval $(0,c]$ for some $c>0$. As such we define the orders of vanishing at $0$ as:
$$
r_{2,S}:=r_{2,S}^+(0), \quad r_{2,I}:=r_{2,I}^+(0), \quad \overline{r}_{2+\alpha,S}:=\overline{r}_{2+\alpha,S}^+(0).
$$
As a remark, if we extend the definition of $\overline{r}_{2+\alpha,S}$ to all $\alpha\in [0,1]$, then
$$
\overline{r}_{2+\alpha,S}+\alpha
$$ is nonincreasing in $\alpha$, and
$$
 \overline{r}_{2+1,S}\geq r_{3,S},
$$
when defined, by definition of $|\varphi''|_{C^{0,\alpha}(\Omega)}$ and the mean value theorem, respectively. Hence, for the purpose of assumption (2) in Theorem \ref{main} and Prop \ref{mainprop}, the following statements are ordered from strongest to weakest (with assumption (2) using (iii)):
\\
(i) $r_{3,S}=r_{2,S}-1$
\\
(ii) $\overline{r}_{2+\beta,S}=r_{2,S}-\beta$ for some $\beta\in (0,1)$
\\
(iii) $\overline{r}_{2+\beta,S}\geq r_{2,S}-\beta+o(\beta)$, for $\beta>0$ small.
\\
\\
Given the unit ball $B = B(0,1)$ in $\R^2$, denote by $\omega_B:\R^2\to\R$ the weight function
$$
\omega_B(x)=\Big(1+{\|x\|}\Big)^{-200}.
$$
For each rectangle $R$ in $\R^2$ with sides parallel to the coordinate axes, let $\omega_R$ be the adaptation of $\omega_B$ to $R$. To be more precise, if $R$ is the rectangle centered at $x_0$ of size $a\times b$ with its sides parallel to the coordinate axes, we define
\[\omega_R(x)=\omega_B(T_R(x-x_0)),\]
where $T_R$ is the linear transformation that maps $(a,0)$ to $(1,0)$ and  $(0,b)$ to $(0,1)$.

\begin{theorem}\label{main}Let $Z$ be a finite subset of $[0,1]$.
	Let $\varphi''$ be never zero and locally H\"older continuous in $[0,1]\setminus Z$ with exponent $\alpha>0$. Furthermore, assume that for each $z\in Z$, $\varphi$ satisfies the following conditions:
	\begin{enumerate}
		\item $r^+_{2,S}(z)=r^+_{2,I}(z)\in (-1,\infty)$ and $r^-_{2,S}(z)=r^-_{2,I}(z)\in (-1,\infty)$
		\item $\overline{r}^+_{2+\beta,S}(z)\geq r^+_{2,S}(z)-\beta+o(\beta)$, and likewise for $r^-$, with $\beta\in (0,\alpha]$.
	\end{enumerate}
	Then, defining the maximal order of vanishing $r$ as
	$$
	r:=\max_{z\in Z}\{r^{+}_{2,S}+2,r^{-}_{2,S}+2,2\},
	$$
	we have, for all $2\le p\le 6$, all $\epsilon>0$, and all $g\in L^1([0,1])$,
	$$
	\|E_{[0,1]}g\|_{L^p(\omega_{R_{\delta,r}})}\le C_\epsilon \delta^{-\epsilon}\left(\sum_{\Delta\in \ptt_{\delta^{1/2}}([0,1])}\|E_\Delta  g\|^2_{L^p(\omega_{R_{\delta,r}})}\right)^{1/2}
	$$
	 where $R_{\delta,r}$ is a rectangle with sides parallel to the coordinate axes of size $\delta^{-1}\times\delta^{-r/2}$, and
	$$
	E_\Delta  g(x_1,x_2):=\int_\Delta g(t)e(t x_1+\varphi(t)x_2)dt.
	$$
\end{theorem}
A simple corollary is the following:
\begin{corollary}
	If $\varphi$ is analytic on $[0,1]$ with curvature not identically zero, then for all $2\leq p \leq 6$ we have
	$$
	\|E_{[0,1]}g\|_{L^p(\omega_{R_{\delta,r}})}\le C_\epsilon \delta^{-\epsilon}\left(\sum_{\Delta\in \ptt_{\delta^{1/2}}([0,1])}\|E_\Delta  g\|^2_{L^p(\omega_{R_{\delta,r}})}\right)^{1/2}
	$$
	for all $g\in L^1([0,1])$, where $R_{\delta,r}$ is a rectangle with sides parallel to the coordinate axes of size $\delta^{-1}\times\delta^{-r/2}$, with $r-2$ being the maximum order of vanishing of $\varphi''$ over the whole curve, and
	$$
	E_\Delta  g(x_1,x_2):=\int_\Delta g(t)e(t x_1+\varphi(t)x_2)dt.
	$$

	\end{corollary}

Our results naturally extend to curves in  general forms.
Let $S$ be a regular curve in $\R^2$ parametrized by $\{(\varphi_1(t),\varphi_2(t)); 0\leq t\leq 1\}$ where $\varphi_i$'s
are $C^\infty$ smooth functions. We assume that the Wronskian of $(\varphi_1',\varphi_2')$ only vanishes at finitely many points, and to finite order. We define, for any subinterval $\Delta$ of $[0,1]$ and for any function $g\in L^1([0,1])$,
$$
E^S_{\Delta}g(x_1,x_2)=\int_\Delta g(t)e(x_1\varphi_1(t)+x_2\varphi_2(t))dt.
$$

\begin{corollary}\label{general} Let $S$ be a regular curve as described above.
For each $\epsilon>0,\,0<\delta\leq 1 $ the following holds:
$$
\|E^S_{[0,1]}g\|_{L^p(\omega_{B_{\delta,r}})}\le C_\epsilon \delta^{-\epsilon}\left(\sum_{\Delta\in \ptt_{\delta^{1/2}}([0,1])}\|E^S_\Delta g\|^2_{L^p(\omega_{B_{\delta,r}})}\right)^{1/2},
$$
where $B_{\delta,r}$ is a ball of radius $\delta^{-r/2}$, and where $r-2$ is the maximum order of vanishing of the Wronskian of $(\varphi_1',\varphi_2')$.
\end{corollary}

Our paper is organized as follows. In section 2 we present a proof for the model case. In section 3, we prove our main theorem, Theorem \ref{main}, and its corollary.

{\bf Acknowledgment.}{ This problem was suggested to the authors as part of the Mathematics Research Community program in June of 2018. The authors would like to thank  Prof. Philip T. Gressman, Prof. Larry Guth, and Prof. Lillian B. Pierce for organizing the  MRC program, suggesting this problem and their constant support. The authors also want to thank Prof. Shaoming Guo, Prof. Yumeng Ou and Prof. Po Lam Yung for many helpful conversations during the program. }

{\bf Notation.} 

$\lesssim_\epsilon$: For nonnegative numbers $A, B$, $A\lesssim_\epsilon B$ means that $A \le C_\epsilon B$ for some constant $C_\epsilon$ which depends on $\epsilon$.
 
$K_p^{par}(\delta)$: We shall use $K_p^{par}(\delta)$ to denote the decoupling constant at scale $\delta$ associated to the standard parabola $(t,t^2),\ t\in[0,1]$ for exponent $p$. 

Rectangles: For a rectangle of side lengths $A$ and $B$, it will be understood that the associated sides are parallel to the first and second coordinate axes, respectively.
\section{Proof for the model case $(t,t^{1+\nu}), \nu>0$}

Fix an $\nu\in (0,\infty)$ and consider the compact curve $\{\gamma(t)=(t,t^{1+\nu})\subset\R^2, t\in[0,1]\}$. We have the following result which is stronger than Theorem \ref{model}:
\begin{proposition}For any $g\in L^1([0,1])$,
	\begin{equation}
	\norm{E_{[0,1]}g}_{L^p(\omega_{R_{\delta,r}})}\lesssim_{\ep}
	\delta^{-\ep}\left(\sum_{\Delta\in\ptt_{\delta^{1/2}}([0,1])}\norm{E_{\Delta}g}_{L^p(\omega_{R_{\delta,r}})}^2\right)^{1/2}
	\end{equation}
	holds for $2\le p\le 6$,  where $R_{\delta,r}$ is a rectangle of side lengths $\delta^{-1}$ and $\delta^{-r/2}$, with $r=\max\{1+\nu,2\}$.
\end{proposition}

\begin{proof}
	Given $0<\ep\ll1$. Decompose the unit interval into
	$$
	[0,1]=\left[0,\delta^{1/2-\ep}\right]\cup\bigcup_{k=1}^K\left[2^{k-1}\delta^{1/2-\ep},2^{k}\delta^{1/2-\ep}\right].
	$$
	Note that we can afford logarithmic losses in $\delta$ and the number of $k$'s is $O(\log(\delta^{-1}))$, so it suffices to show that for any $\delta^{1/2-\ep}\le a\le 1/2$,
	\begin{equation}\label{EaBig}
	\norm{E_{[a,2a]}g}_{L^p(\omega_{R_{\delta,r}})}\lesssim_{\ep}
	\delta^{-\ep}\left(\sum_{\Delta\in\ptt_{\delta^{1/2}}([a,2a])}\norm{E_{\Delta}g}_{L^p(\omega_{R_{\delta,r}})}^2\right)^{1/2}.
	\end{equation}
	We claim that for any $a\in [\delta^{1/2-\ep},1/2]$, we have the following inequality:
	\begin{equation}\label{Ea}
	\norm{E_{[a,2a]}g}_{L^p(\omega_{{R_{a,\nu,\delta}}})}
	\lesssim_{\ep}
	\delta^{-\ep}
	\left(\sum_{\Delta\in\ptt_{\delta^{1/2}}([a,2a])}\norm{E_{\Delta}g}_{L^p(\omega_{R_{a,\nu,\delta}})}^2\right)^{1/2},
	\end{equation}
	where ${R_{a,\nu,\delta}}$ is a rectangle of size $\delta^{-1}\times a^{1-\nu}\delta^{-1}$.
	
	Once we prove \eqref{Ea}, \eqref{EaBig} follows by Minkowski's inequality and the fact that $a\le 1$.
	Given $[a,2a]$, let $t_0\in [a,2a]$. The Taylor expansion
	$$
	a^{1-\nu}(t_0+\Delta t)^{1+\nu}
	=a^{1-\nu}(t_0^{1+\nu}+(1+\nu)t_0^{\nu}\Delta t+\frac{(1+\nu)\nu}{2}t_0^{\nu-1}(\Delta t)^2)
	+a^{1-\nu}t_0^{\nu-2}O((\Delta t)^3)
	$$
	shows that if $\Delta t\le\delta^{1/2-\sigma}$, where $\sigma=\frac{\ep}{3}$, then on the interval $[t_0,t_0+\Delta t]$, the curve
	$\gamma_a(t):=(t,a^{1-\nu}t^{1+\nu})$ is within $\delta$ from the parabola
	\begin{equation}\label{parabola}
	\left(t,a^{1-\nu}(t_0^{1+\nu}+(1+\nu)t_0^{\nu}t+\frac{(1+\nu)\nu}{2}t_0^{\nu-1}t^2)\right).
	\end{equation}
	In fact, since $a\ge\delta^{1/2-\ep}$, the error
	$\abs{a^{1-\nu}t_0^{\nu-2}(\Delta t)^3}\sim\abs{a^{-1}(\Delta t)^3}\le\delta$.
	Define
	$$
	E_{\Delta,\gamma_a}g(x_1,x_2)=\int_{\Delta}g(t)e(t x_1+a^{1-\nu}t^{1+\nu}x_2)dt,
	$$
	then
	$$
	E_{\Delta}g(x_1,x_2)=E_{\Delta,\gamma_a}g(x_1,a^{\nu-1}x_2),
	$$
	and thus
	\begin{equation}
	\norm{E_{\Delta}g}_{L^p({R_{a,\nu,\delta}})}
	=a^{\frac{1-\nu}{p}}\norm{E_{\Delta,\gamma_a}g}_{L^p(Q_{\delta})},
	\end{equation}
 where $Q_{\delta}$ is a cube of side length $\delta^{-1}$.
	So to prove the claim, it suffices to show
	\begin{equation}
	\norm{E_{[a,2a],\gamma_a}g}_{L^p(\omega_{Q_{\delta}})}
	\lesssim_{\ep}
	\delta^{-\ep}
	\left(\sum_{\Delta\in\ptt_{\delta^{1/2}}([a,2a])}\norm{E_{\Delta,\gamma_a}g}_{L^p(\omega_{Q_{\delta}})}^2\right)^{1/2}.
 	\end{equation}
	This is equivalent to showing the smallest constant $K_p(\delta)$ that makes the following inequality hold satisfies $K_p(\delta)\lesssim_{\ep}\delta^{-\ep}$:
\begin{equation} \label{it}
	\norm{f}_{L^p(\omega_{Q_{\delta}})}\le K_p(\delta)\Bigg(\sum_{\theta\in P_{\delta}}\norm{f_{\theta}}^2_{L^p(\omega_{Q_{\delta}})}\Bigg)^{1/2},
\end{equation}
	where $\supp \hat{f}$ is contained in $\mathscr{N}_{\delta}(\gamma_a)$, the $\delta$ neighborhood of the curve $\{\gamma_a(t),a\le t\le2a\} $, and $P_{\delta}$ is a finitely overlapping cover of $\mathscr{N}_{\delta}$ with curved regions $\theta$ of the form
	$$
	\theta=\{(t,\eta+a^{1-\nu}t^{1+\nu}):t\in I_{\theta},\abs{\eta}\le 2\delta\},
	$$
	where $I_{\theta}$ runs over all intervals with length $\delta^{1/2}$ and center belongs to $\delta^{1/2}\mathbb Z\cap[a,2a]$. Note that by Minkowski, \eqref{it} implies estimates with the same constant over any spatial cubes with side length larger than $\delta^{-1}$, and thus we shall always use the weight associated to the largest spatial cube throughout our iteration.
	
	We apply the iteration argument sketched in \cite{bourgain2015proof}, \cite{garrigos2010mixed} and \cite{Pramanik2007}. Namely, we have
	$$
	\norm{f}_{L^p(\omega_{Q_{\delta}})}\le K_p(\delta^{1-\sigma})\left(\sum_{\tau\in P_{\delta^{1-\sigma}}}\norm{f_{\tau}}^2_{L^p(\omega_{Q_{\delta}})}\right)^{1/2}.
	$$
	And the decoupling inequality for the parabola \eqref{parabola} implies that
	$$
	\norm{f_{\tau}}_{L^p(\omega_{Q_{\delta}})}\le K^{par}_p(\delta)\left(\sum_{\theta\in P_{\delta},\theta\subset\tau}\norm{f_{\theta}}^2_{L^p(\omega_{Q_{\delta}})}\right)^{1/2},
	$$
	where $K^{par}_p$ is the standard $(t,t^2)$ paraboloid decoupling constant with $$K_p^{par}(\delta)\lesssim_{\ep}\delta^{-\ep}.$$
	
	We iterate to get
	$$
	K_p(\delta)\le C_{\ep}^k\delta^{-\ep(1-(1-\sigma)^k)}K_p(\delta^{(1-\sigma)^k}),
	$$
	where $k$ should be chosen so that $\delta^{(1-\sigma)^k}\sim a^2 \leq 1/4$. From this it follows that
	$K_p(\delta)\le C_{\ep}\delta^{-\ep}$, with the constant uniform in $a$. A more detailed treatment of this iteration process will be given in the next section.
\end{proof}

\section{Proof of Theorem \ref{main}}
The work of Bourgain-Demeter implies the following:

\begin{lemma}
	Let $\varphi\in C^{2,\alpha}[0,1]$ satisfy $\varphi''>0$. Then we have, for all $2\leq p \leq 6$,
	$$
	\|E_{[0,1]}g\|_{L^p(\omega_{Q_{\delta}})}\lesssim_\epsilon \delta^{-\epsilon}\left(\sum_{\Delta\in \ptt_{\delta^{1/2}}([0,1])}\|E_\Delta  g\|^2_{L^p(\omega_{Q_{\delta}})}\right)^{1/2}
	$$
	for all $g\in L^1([0,1])$, where $Q_{\delta}$ is a cube of side length $\delta^{-1}$.
	
\end{lemma}

To prove Theorem \ref{main}, it will suffice to prove the following proposition.

\vspace{.5pc}

First, recall the simplified definitions:
$$ r_{2,S}=\sup\Big\{s:\lim_{t\rightarrow 0^+}\frac{\varphi''(t)}{t^s}=0\Big\},
\quad \quad
r_{2,I}=\inf\Big\{s:\lim_{t\rightarrow 0^+}\Big|\frac{\varphi''(t)}{t^s}\Big|=\infty\Big\},
$$
and, defining
$\varphi^{2,\alpha}(t):=|\varphi''|_{C^{0,\alpha}([t,2t])},
\quad
\overline{r}_{2+\alpha,S}=\sup\{s:\lim_{t\rightarrow 0^+}\frac{\varphi^{2,\alpha}(t)}{t^s}=0\}.$

\begin{proposition}\label{mainprop}
	Let $\varphi''$ be positive and locally $\alpha$-H\"older continuous in $(0,\tilde{c}]$, for some $\alpha, \tilde{c}>0$. Also, let the orders of vanishing of $\varphi''$ satisfy
	\begin{enumerate}
\item	$r_{2,S}=r_{2,I}=:r_2\in (-1,\infty)$
\item   $\overline{r}_{2+\beta,S}\geq r_2-\beta+o(\beta)$, for $\beta\in (0,\alpha]$.	
	\end{enumerate} Then we have, for all $2\leq p \leq 6$ and with $r:=\max\{r_2+2,2\}$,
	$$
	\|E_{(0,c]}g\|_{L^p(\omega_{R_{\delta,r}})}\lesssim_\epsilon \delta^{-\epsilon}\left(\sum_{\Delta\in \ptt_{\delta^{1/2}}((0,c])}\|E_\Delta  g\|^2_{L^p(\omega_{R_{\delta,r}})}\right)^{1/2}
	$$
	for all integrable $g: (0,c]\rightarrow \C$, where $R_{\delta,r}$ is a rectangle of side lengths $\delta^{-1}$ and $\delta^{-\frac r2}$, and $c$ is chosen sufficiently small, independent of $\delta$, $\epsilon$, or $\alpha$.
	
\end{proposition}

\begin{proof} Decompose interval $(0,c]$ into:
$$
(0,c]=\left(0,\delta^{1/2-\epsilon}\right]\cup\bigcup_{k=1}^{K}\left[2^{k-1}\delta^{1/2-\epsilon}, 2^k\delta^{1/2-\epsilon}\right].
$$
We automatically get the desired decoupling on the $(0,\delta^{\frac 12-\epsilon}]$ interval. Since we can afford log losses in $\delta$ and the number of $k$'s is $O(\log(\delta^{-1}))$, it suffices to show that for any $\delta^{1/2-\epsilon}<a<c$,
\begin{equation}\label{Eg1}
\|E_{[a,2a]}g\|_{L^p(\omega_{R_{\delta,r}})}\lesssim_\epsilon \delta^{-\epsilon}\left(\sum_{\Delta\in \ptt_{\delta^{1/2}}([a,2a])}\|E_\Delta  g\|^2_{L^p(\omega_{R_{\delta,r}})}\right)^{1/2}.
\end{equation}
Define $$\varphi''_a:=\min\{\varphi''(t):t\in [a,2a]\}.$$ By hypothesis (1), there exists some constant $D_{\epsilon}>0$ such that
\begin{equation}\label{lowerbd_phi2}
\varphi''_a\geq D_{\epsilon}^{-1}a^{r_2+\frac{\epsilon r_2}{100}}\geq D_{\epsilon}^{-1}\delta^\frac{r_2}{2}
\end{equation}
for $\delta^{\frac 12-\epsilon}<a<c$, $c$ sufficiently small. Defining $R_{\delta,\varphi''_a}$ as a rectangle with side lengths $\delta^{-1}$ and $ \delta^{-1}(\varphi''_a)^{-1}$, we claim that
\begin{equation}\label{Eg2}
\|E_{[a,2a]}g\|_{L^p(\omega_{R_{\delta,\varphi''_a}})}\lesssim_\epsilon \delta^{-\epsilon}\left(\sum_{\Delta\in \ptt_{\delta^{1/2}}([a,2a])}\|E_\Delta  g\|^2_{L^p(\omega_{R_{\delta,\varphi''_a}})}\right)^{1/2}.
\end{equation}
By Minkowski and \eqref{lowerbd_phi2}, \eqref{Eg2} will imply \eqref{Eg1} if $c$ is chosen sufficiently small (where, due to our use of weighted norms and the definition of $\omega_R$, the extra $D_\epsilon$ can be absorbed into the implicit $\lesssim_\epsilon$ constant).
\\

 Let $t_0\in [a,2a]$, and define the curve $\gamma_{a}(t)=\left(t,\frac{2\varphi(t)}{\varphi''_a}\right)$. Similarly, define the paraboloid
\begin{equation}\label{rho1}
\rho_{a,t_0}(t)=\left(t,\frac{2\varphi(t_0)}{\varphi''_a}+\frac{2\varphi'(t_0)}{\varphi''_a}(t-t_0)+\frac{\varphi''(t_0)}{\varphi''_a}(t-t_0)^2\right).
\end{equation}
Then, for $0\leq \Delta t\lesssim t_0$,
\begin{align*}
|\gamma_{a}(t_0+\Delta t)-\rho_{a,t_0}(t_0+\Delta t)|&\lesssim \Bigg{|}\frac{\varphi^{2,\beta}(t_0)}{\varphi''_a}\Bigg{|}(\Delta t)^{2+\beta}\leq C_{\epsilon} \frac{a^{\overline{r}_{2+\beta,S}-\frac{\epsilon \beta}{27}}}{a^{r_2+\frac{\epsilon \beta}{27}}}(\Delta t)^{2+\beta}
\\
&\leq C_{\epsilon} a^{-(\beta+\frac{\epsilon \beta}{9})}(\Delta t)^{2+\beta}
\end{align*}
by Taylor's Theorem, hypothesis on $r_2$, and hypothesis on $\overline{r}_{2+\beta,S}$ for a proper choice of $\beta(\epsilon)$. Namely, since $\overline{r}_{2+\beta,S}\geq r_2-\beta+o(\beta)$ by assumption (2), there exists a choice of $\beta(\epsilon)$ that vanishes at $0$ sufficiently quickly such that $\overline{r}_{2+\beta,S}\geq r_2-\beta(\epsilon)-\frac{\epsilon \beta(\epsilon)}{27}$.

\vspace{.5pc}

Letting $\Delta t_{max}=\sup\{s: |(\gamma_{a}-\rho_{a,t_0})(t_0+\Delta t)|<\delta \text{ for all } 0<\Delta t<s\} $ and recalling that $a\geq \delta^{1/2-\epsilon}$, then either $\Delta t_{max}>a$ or, for some constant $c_\epsilon>0$,
\begin{equation}\label{tmax}
\Delta t_{max}\geq c_{\epsilon} \delta^{\frac{1}{2+\beta}}a^\frac{\beta+\frac{\epsilon \beta}{9}}{2+\beta}\geq c_{\epsilon} \delta^\frac 1{2+\beta} \delta^\frac{\frac{\beta}{2}}{2+\beta}\delta^\frac{\frac{\epsilon \beta}{18}-\epsilon\beta}{2+\beta}\geq c_{\epsilon} \delta^{\frac 12-\frac{\epsilon \beta}{4}}.
\end{equation}
Therefore, on $[t_0,\min\{2a,t_0+c_{\epsilon}\delta^{\frac 12-\frac{\epsilon\beta}{4}}\}]$, $(t,\rho_{a,t_0}(t))\in \scriptN_\delta(\gamma_a)$.

\vspace{.5pc}

Defining
	$$
	E_{\Delta,\gamma_a}g(x_1,x_2)=\int_{\Delta}g(t)e(t x_1+\tfrac{2\varphi(t)}{\varphi''_a}x_2)dt,
	$$
	then
	$$
	E_{\Delta}g(x_1,x_2)=E_{\Delta,\gamma_a}g(x_1,\tfrac{\varphi''_a}{2}x_2),
	$$
	and thus
	\begin{equation*}
	\norm{E_{\Delta}g}_{L^p({\omega_{R_{\delta,\varphi''_a}}})}
	=\left(\tfrac{2}{\varphi''_a}\right)^{1/p}\norm{E_{\Delta,\gamma_a}g}_{L^p(\omega_{Q_{\delta}})},
	\end{equation*}
 where $Q_{\delta}$ is a cube of side length $\delta^{-1}$. Thus to prove the claim, it suffices to show
	\begin{equation*}
	\norm{E_{[a,2a],\gamma_a}g}_{L^p(\omega_{Q_{\delta}})}
	\lesssim_{\ep}
	\delta^{-\ep}
	\left(\sum_{\Delta\in\ptt_{\delta^{1/2}}([a,2a])}\norm{E_{\Delta,\gamma_a}g}_{L^p(\omega_{Q_{\delta}})}^2\right)^{1/2}.
	\end{equation*}
	This is equivalent to showing the smallest constant $K_p(\delta)$ that makes the following inequality holds satisfies $K_p(\delta)\lesssim_{\ep}\delta^{-\ep}$: (noting that $K_p(\delta)$ is tied to $\delta^{1/2}$ partitions)
	$$
	\norm{f}_{L^p(\omega_{Q_\delta})}\le K_p(\delta)\left(\sum_{\theta\in P_{\delta}}\norm{f_{\theta}}^2_{L^p(\omega_{Q_\delta})}\right)^{1/2},
	$$
	where $\supp \hat{f}$ is contained in $\mathscr{N}_{\delta}(\gamma_a)$, the $\delta$ neighborhood of the curve $\{\gamma_a(t),a\le t\le2a\} $, and $P_{\delta}$ is a finitely overlapping cover of $\mathscr{N}_{\delta}$ with curved regions $\theta$ of the form
	$$
	\theta=\{(t,\eta+\gamma_a(t)):t\in I_{\theta},\abs{\eta}\le 2\delta\},
	$$
	where $I_{\theta}$ runs over all intervals with length $\delta^{1/2}$ and center belongs to $\delta^{1/2}\mathbb Z\cap[a,2a]$.
	
\vspace{1pc}	
	
	We apply the iteration argument sketched in \cite{bourgain2015proof}, \cite{garrigos2010mixed} and \cite{Pramanik2007} Namely, we have
	$$
	\norm{f}_{L^p(\omega_{Q_\delta})}\le K_p(c_\epsilon^2 \delta^{1-\epsilon\beta/2})\left(\sum_{\tau\in P_{\delta^{1-\epsilon\beta/2}}}\norm{f_{\tau}}^2_{L^p(\omega_{Q_\delta})}\right)^{1/2}.
	$$
	And due to \eqref{tmax}, we can use the decoupling inequality for the parabola $\rho_{a,t_0}$ in \eqref{rho1} to obtain
	$$
	\norm{f_{\tau}}_{L^p(\omega_{Q_\delta})}\le K^{par}_p(\delta)\left(\sum_{\theta\in P_{\delta},\theta\subset\tau}\norm{f_{\theta}}^2_{L^p(\omega_{Q_\delta})}\right)^{1/2},
	$$
	where $K^{par}_p$ is the standard $(t,t^2)$ paraboloid decoupling constant with $$K^{par}_p(\delta)\leq C_\sigma \delta^{-\sigma}$$
	for all $\sigma>0$. Note that $\rho_{a,t_0}$ has constant curvature $\frac{\varphi''(t_0)}{\varphi''_a}\gtrsim 1$, and due to rescaling and an application of Minkowski,  $K^{par}_p$ still bounds the decoupling for $\rho_{a,t_0}$. More specifically, normalizing the curvature of $\rho_{a,t_0}$ will expand the spatial rectangles while leaving $K_p$ unchanged. Finally, an application of Minkowski implies that decoupling constants $K_p$ over larger spatial rectangles are bounded by the decoupling constants over smaller rectangles.
	
	\vspace{.5pc}
	
	Using $\epsilon^2\beta/ 4$ in place of $\sigma$, this implies that
	$$K_p(\delta)\leq C_{\epsilon^2\beta/4} \delta^{-\epsilon^2\beta/4} K_p(c_\epsilon^2 \delta^{1-\epsilon\beta/2})\leq c_\epsilon^{-1}C_{\epsilon^2\beta/4} \delta^{-\epsilon^2\beta/4} K_p(\delta^{1-\epsilon\beta/2}).
$$
For the second inequality, we used the fact that $K_p(\sigma_1^2\sigma_2^2)\leq \sigma_1^{-1}K_p(\sigma_2^2)$. Denote $\overline{C}_\epsilon=c_\epsilon^{-1}C_{\epsilon^2\beta/4}$.
	Using $1+(1-\frac{\epsilon\beta}{2})+...+(1-\frac{\epsilon\beta}{2})^{k-1}=\frac{2}{\epsilon\beta}(1-(1-\frac{\epsilon\beta}{2})^k)$, we iterate to get
	\begin{align*}
	K_p(\delta)&\le \overline{C}_\epsilon^k\delta^{-\frac{\epsilon}{2}(1-(1-\frac{\epsilon\beta}{2})^k)}K_p(\delta^{(1-\frac{\epsilon\beta}{2})^k})
	\\
	&\leq \overline{C}_\epsilon^k\delta^{-\frac{\epsilon}{2}}K_p(\delta^{(1-\frac{\epsilon\beta}{2})^k}),
	\end{align*}
	where $k$ should be chosen so that $\delta^{(1-\frac{\epsilon\beta}{2})^k}\sim a^2$. Replacing $a^2$ with $e^{-1}$, we get
$$
k\leq \frac{\log \log (\delta^{-1/2})}{-\log(1-\epsilon\beta/2)},
$$	
so
$$
\overline{C}_\epsilon^k\leq (\log (\delta^{-1/2}))^\frac{\overline{C}_\epsilon}{\log(1-\epsilon\beta/2)}\leq C_\epsilon \delta^{-\epsilon/2}
$$	
for any choice of $\beta(\epsilon)$.
	From this it follows that
	$K_p(\delta)\le C_{\ep}\delta^{-\ep}$, with the constant uniform in $a$. (Note that $K_p(e^{-1})\approx 1$.)
\end{proof}

Now we finish the paper with a simple proof of Corollary \ref{general} and a corresponding discretized version. The latter might be useful for future applications which gives bounds similar to those of Wooley \cite{Wooley}.
\begin{proof}[Proof of Corollary \ref{general}] The Wronskian of $(\varphi_1',\varphi_2')$ is:
		\[\det\left[\begin{matrix}
		\varphi_1'& \varphi_2'\\
		\varphi_1''& \varphi_2''
		\end{matrix}\right]=\varphi_1'\varphi_2''-\varphi_2'\varphi_1''.\] Given $t_0$, since $S$ is a regular curve, we may assume $\varphi'_1(t)\neq0$ near $t_0$. Using a partition of unity, it suffices to consider a piece of the curve containing $t_0$ where $\varphi'_1(t)\neq0$.  When we do the change of variable $s=\varphi_1(t)$, then the curve becomes
		\[(s,\varphi_2(\varphi_1^{-1}(s)))=(s,\psi(s)),\]
		where $t=\varphi^{-1}_1(s)$, and $$\psi'(s)=\frac{\varphi'_2(t)}{\varphi'_1(t)}=\frac{\varphi'_2(\varphi^{-1}_1(s))}{\varphi'_1(\varphi^{-1}_1(s))},$$
		with a non-zero denominator by our assumption. Then the second derivative equals
		$$\psi''(s)=\frac{[\varphi_1'\varphi_2''-\varphi_2'\varphi_1''](t)}{[\varphi'_1(t)]^2}=\frac{[\varphi_1'\varphi_2''-\varphi_2'\varphi_1''](\varphi^{-1}_1(s))}{[\varphi'_1(\varphi^{-1}_1(s))]^2}=\frac{\text{Wronskian at }t}{[\varphi'_1(t)]^2},$$
	which satisfies the requirement of Theorem \ref{main}, with the order of vanishing of the curvature equaling the order of vanishing of the Wronskian, if the Wronskian only vanishes of finite order at a finite number of points.
		
	\end{proof}
	
Corollary \ref{general} gives the following discretized version.
\begin{theorem}\label{discretized}
	If  $\frac{n-1}{N}<t_n\leq\frac{n}{N}$ for each $1\leq n\leq N$ is a collection of points in $[0,1]$, then for each $R\geq N^r$, $(\varphi_1,\varphi_2)$ as in Corollary \ref{general}, we have
	$$
	\left(\frac{1}{R^2}\int\Bigg|\sum_{n=1}^Na_ne(x_1\varphi_1(t_n)+x_2\varphi_2(t_n))\Bigg|^6\,w_{B_R}(x)dx_1dx_2\right)^{1/6}\hspace{-.5cm}\lesssim_\epsilon N^\epsilon\left(\sum_{n=1}^N|a_n|^2\right)^{1/2}.
	$$
	Here $B_R$ denotes a ball of radius $R$.
\end{theorem}
\begin{proof}
The proof is along the same line of proof of Theorem~{4.1} in \cite{bourgain2016proof}. For completeness we give the details of the proof. Let  $\{B_{N^r}\}$ be a collection of finitely overlapping balls covering $B_R$. We apply Corollary~\ref{general} to each $B_{N^r}$ and use the fact that
$$
\sum w_{B_{N^r}}\leq C w_{B_R}
$$
to get
$$
\|E_{[0,1]}g\|_{L^6(w_{B_R})}\lesssim_\epsilon N^{\epsilon}\left(\sum_{\Delta\in \ptt_{N^{-1}}([0,1])}\|E_\Delta  g\|^2_{L^6(w_{B_R})}\right)^{1/2}.
$$
For $\tau>0$ let $g_\tau(t)=\frac{1}{2\tau}\sum_{n=1}^Na_n\chi_{[t_n-\tau,t_n+\tau]}(t)$ and apply the above inequality and then let $\tau$ to $0$. This completes the proof.	
\end{proof}

\bibliographystyle{plain}
\bibliography{decoupling}

\end{document}